\documentclass[11pt]{article}

\usepackage{amsmath,amsthm,amsfonts,amssymb,amsopn,amscd, mathrsfs}
\usepackage{latexsym}
\usepackage{graphics}

\DeclareMathOperator{\fn}{null}

\def\be{\begin{equation}}
\def\ee{\end{equation}}

\def\ben{\begin{equation*}}
\def\een{\end{equation*}}

\def\RR{{\mathbb R}}
\def\CC{{\mathbb C}}
\def\NN{{\mathbb N}}
\def\ZZ{{\mathbb Z}}
\def\TT{{\mathbb T}}

\def\isom{\cong}

\def\a{\alpha}

\def\g{\gamma}

\def\A{{\cal A}}

\def\H{{\cal H}}

\def\S{{\cal S}}

\def\emptyset{\varnothing}

\def\Diff{{\rm Diff}}

\def\S2{S^{1(2)}}

\def\fg{{\mathfrak g}}
\def\fh{{\mathfrak h}}
\def\fn{{\mathfrak n}}
\def\fgc{\mathfrak{g}_\CC}
\def\<{\langle}
\def\>{\rangle}
\def\sg{\mathscr{S}\fgc}
\def\sch{\mathscr{S}}
\def\dg{\mathscr{D}\fgc}

\def\supp{\mathrm{supp}}

\def\a{\alpha}

\def\g{\gamma}

\newtheorem{theorem}{Theorem}[section]
\newtheorem{lemma}[theorem]{Lemma}

\newtheorem{corollary}[theorem]{Corollary}

\theoremstyle{definition}
\newtheorem{definition}[theorem]{Definition}

\theoremstyle{remark}
\newtheorem{remark}[theorem]{Remark} 

\title{Ground state representations of loop algebras}
\author{
{\bf Yoh Tanimoto\footnote{Supported in part by the ERC Advanced Grant 227458
OACFT ``Operator Algebras and Conformal Field Theory''.}}\\
Dipartimento di Matematica, Universit\`a di Roma ``Tor
Vergata''\\ Via della Ricerca Scientifica, 1 - I--00133 Roma, Italy.\\
E-mail: {\tt tanimoto@mat.uniroma2.it}}

\begin{document}
\maketitle
\begin{abstract}
Let $\fg$ be a simple Lie algebra, $L\fg$ be the loop algebra of $\fg$.
Fixing a point in $S^1$ and identifying the real line with the punctured circle,
we consider the subalgebra $\sch \fg$ of $L\fg$ of rapidly decreasing elements on $\RR$.
We classify the translation-invariant 2-cocycles on $\sch\fg$.
We show that the ground state representation of $\sch\fg$ is unique for each
cocycle. These ground states correspond precisely to the vacuum representations
of $L\fg$.
\end{abstract}

\section{Introduction}
For a compact connected Lie group $G$, the group of smooth maps from the circle $S^1$
to $G$ is called the loop group $LG$ of $G$. Loop groups have been a subject
of extensive research both from purely mathematical and physical viewpoints
(\cite{PS}, \cite{Wassermann}, \cite{Xu}, \cite{GF}, \cite{Toledano}, \cite{DMS}).
On the one hand,
the representation theory of $LG$ has a particularly simple structure. If we consider
positive energy projective representations (defined below), and if $G$ is simply connected,
then such representations behave very much like ones of compact groups. They are
completely reducible, irreducible representations are classified by their ``lowest weights'',
and irreducible representations are realized as
the spaces of complex line bundles on the group by analogy with Borel-Weil theory
\cite{PS}. On the other hand, any such representation can be considered as a
charged sector of a conformal field theory.

It is a natural variant to think about the group of maps from the real line
$\RR$ into $G$. The natural group of covariance is now the translation group.
Since $S^1$ is a one-point compactification of $\RR$, we consider
this group as a subgroup of $LG$.
Then one would expect that there should arise several representations which
do not extend to $LG$. This problem has been open for a long time \cite{PS}.

The main objective of this paper is to show the contrary
at the level of Lie algebra with the assumption of existence of an invariant vector:
Namely, if a (projective) unitary representation of $\sch\fgc$ (the Lie subalgebra
of $L\fgc$ of Schwartz class elements, defined below) is covariant with respect to
translation and admits a cyclic vector invariant under translation, then
it extends to a representation of $L\fgc$.
Then even a complete classification of such representations with a
``ground state vector'' follows due to the classification for $L\fgc$ by
Garland \cite{Garland} or at group level by Pressley and Segal \cite{PS}.
This is in a clear contrast to the case of the diffeomorphism group on the circle, where
the subgroup of elements which fix the point of infinity has many representations
which do not extend to the whole group \cite{tanimoto1}. Moreover, several new representations
arise when we further restrict the consideration to the subgroup of diffeomorphisms of $\RR$
with compact support, which include representations with an invariant vector with respect to
translations \cite{tanimoto2}.

Besides the interest from a purely representation-theoretic context, the study of positive energy
representations with an invariant vector for translation is motivated by physics, in particular
by chiral conformal field theory. I will explain this aspect based on the operator-algebraic
approach to CFT.

In the setting of algebraic quantum field theory, a chiral component of a conformal field
theory is a net of von Neumann algebras on the circle satisfying natural requirements which come
from physics (isotony, locality, covariance, existence of vacuum, etc.)\cite{GF}. To construct
examples of such nets, we can utilize positive energy representations of loop groups, and
in fact these examples have played a key role in the classification of certain conformal
field theories \cite{KL}, \cite{Xu}.

For a certain class of representations of nets, a sophisticated theory has been established
by Doplicher-Haag-Roberts (for its adaptation to chiral CFT, see \cite{GF}).
The DHR theory is concerned with representations which are
localized in some interval, i.e., unitarily equivalent to the original (vacuum) representations
outside the interval of localization. These representations are considered to describe
the states with finite charge.

On the other hand, in a physical context we are sometimes interested in a larger class of
representations. A typical case occurs in the study of thermal equilibrium states.
A thermal equilibrium state is invariant with time, thus in the context of one-dimensional
chiral theory it is invariant under translation. By physical intuition, we would say that
a state with a finite amount
of charge cannot be invariant under translation. Then we should consider a more general class
of representations. As explained later, an invariant state for translation whose GNS representation
has positive-energy can be considered as an equilibrium state with temperature zero. Physicists
call it a ground state.

Nets of von Neumann algebras generated by representations of loop groups are known to have
a property called complete rationality \cite{KLM}, \cite{Xu}.
This complete rationality implies that
the net has only finitely many inequivalent irreducible DHR representations.
Physically it means only finite amount
of charge is possible in such a model. Then one would guess that any completely rational
net has only equilibrium states without charge. We will prove a result on representations
of the Lie algebras of loop groups which strongly supports this point of view, namely, we will show
that any ground state representation of the loop algebra (in a certain sense clarified below)
is the vacuum representation.

Similar lines of research can be conducted also for equilibrium states
with finite temperature \cite{CLTW}, in which the authors have shown
that if a conformal net is completely rational then it admits the
unique KMS state.

John E. Roberts has proved that for a general dilation-covariant net of observables
there is a unique dilation-invariant state, the vacuum \cite{Roberts}.
This in particular tells us that a ground state different from the vacuum
cannot be dilation-invariant (although this never excludes the existence of
other ground states). In fact, the composition of a ground state on the Virasoro nets with dilation
is used to produce different ground states \cite{tanimoto2}. A similar technique is used in \cite{CLTW}
to obtain continuously many different KMS states.

At the end of the introduction, I would like to note that
the above result on KMS states has been
proved with the techniques of operator algebras, in particular subfactors, and utilizes
relationships between several nets. On the other hand,
the present result on the uniqueness of ground states for loop algebras relies only on
elementary facts on Lie algebras and gives a direct proof.

Unfortunately, the present result does not imply directly the uniqueness of ground state
of nets of von Neumann algebras. There are still difficulties in the
differentiability of given representations and extension to ``Schwartz class'' algebra.
These problems will be discussed in the final section.

This paper is organized as follows.
In section 2 we recall standard facts on loop groups $LG$ and loop algebras $L\fg$,
their central extensions and representation theory. In section 3 we introduce the
main object of this paper, the algebra $\sch \fg$. In section 4 we prove that
translation-invariant 2-cocycle on $\sch \fg$ is essentially unique up to scalar.
In section 5 we prove that ground states on $\sch \fg$ can be classified only by
the cocycle. In section 6 we discuss the physical meaning of ground states and
possible implications to the representation theory of conformal nets of
von Neumann algebras.

\section{Preliminary}
In this section, we collect notations and basic results on loop groups and loop algebras.
Throughout this article, $G$ is a simple and simply connected Lie group.
Let $\fg$ be the Lie algebra of $G$.
\subsection{The loop group $LG$}
Although the main results of this paper regard representations of
infinite-dimensional Lie algebras, we would like to start with preliminary
on groups. The representation theory of infinite-dimensional Lie algebras
inevitably contains unbounded operators on infinite dimensional vector spaces,
and standard notions as irreducibility and decomposition of representations could be
difficult or practically not appropriate to define in total generality.
On the other hand, at the group level
the representation theory of compact group gives us a strong device for such
reduction.

The present paper is motivated by the representation theory of the loop group
$LG$ of $G$:
\begin{align*}
LG := \{g:S^1 \to G, \mbox{smooth}\}, \\
g_1\cdot g_2(z) := g_1(z)\cdot g_2(z), z \in S^1.
\end{align*}

This group is an infinite-dimensional Lie group with the Lie algebra $L\fg$ called
loop algebra:
\begin{align*}
L\fg := \{\xi:S^1 \to \fg, \mbox{smooth}\}, \\
[\xi_1,\xi_2](z) := [\xi_1(z),\xi_2(z)], z \in S^1.
\end{align*}
$L\fg$ has the natural topology by the uniform convergence of each derivative
and the differential structure.
It is also possible to define a natural topology on the group $LG$, and
there is a smooth map from the neighbourhood of the unit element of $LG$ to
the neighbourhood of $0$ in $L\fg$. The group operation corresponds to the
bracket \cite{Milnor} \cite{PS}. 

A 2-cocycle on a group $H$ with values in $\TT$ is a map $\g: H\times H \to \TT$ which
satisfies (see \cite[Chapter 3]{Schottenloher})
\[
\g(e,e) = 1, \phantom{...} \g(f,g)\g(fg,h) = \g(f,gh)\g(g,h),
\]
where $e$ is the unit element of $H$. If there is $\beta: H \to \TT$ such that
$\g(f,g) = \frac{\beta(f)\beta(g)}{\beta(fg)}$, then $\g$ is said to be coboundary.
The set of 2-cocycles forms a group by defining the product with pointwise
multiplication. If one cocycle is a multiple of a coboundary with another cocycle,
these two cocycles are said to be equivalent.

A group $\widetilde{LG}$ is called a central extension of $LG$ by $\TT$ if there is
an exact sequence
\[
0 \to \TT \to \widetilde{LG} \to LG \to 0.
\]
A central extension is said to be split if $\widetilde{LG} \isom \TT \times LG$.
There is a one-to-one correspondence between central extensions of $LG$ and
equivalence classes of 2-cocycles \cite{Schottenloher}.

The following is fundamental \cite[Chapter 4]{PS}.
\begin{theorem}[Pressley and Segal]
If $G$ is simple and simply connected, then there exists a family of central
extensions of $LG$ which are parametrized by positive integers, and all such
extensions come from the central extensions of $L\fg$ (see below).
\end{theorem}

The 2-cocycles of the group appear when we consider projective representations.
\begin{definition}
Let $\H$ be a Hilbert space. A map $\pi: LG \to U(\H)$ is a projective
unitary representation if there is a 2-cocycle $\g$ of $G$ such that
\[
\pi(g_1)\pi(g_2) = \g(g_1,g_2)\pi(g_1\cdot g_2).
\]
\end{definition}

If we have a projective representation of $LG$, by definition it also
specifies a 2-cocycle of $LG$, and this cocycle determines the class
of central extensions. We call this class the {\bf level} of the representation.
We can naturally think that the given projective
representation of $LG$ as a ``true'' representation, not projective,
of the central extension.

Note that the circle $S^1$ acts on $LG$ by rotation:
\[
g_\theta(z) := g(e^{-i\theta}z).
\]
\begin{definition}
A projective unitary representation $\pi$ of $LG$ is said to have positive
energy if there is a unitary representation $U$ of $S^1$ on the same
Hilbert space with positive spectrum such that
\[
U(\theta)\pi(g)U(\theta)^* = \pi(g_\theta).
\]
\end{definition}
\begin{remark}\label{invariancerot}
Let us define the action of rotation on the space of 2-cocycles
by $\g_\theta(g_1,g_2) = \g((g_1)_{-\theta},(g_2)_{-\theta})$. Then
any positive energy representation has a 2-cocycle which is
invariant under this action of translation.
\end{remark}

The loop group $LG$ contains constant loops. The set of constant loops forms
a subgroup isomorphic to $G$. A constant loop is of course
invariant under rotation, hence it commutes with the action of rotation.
As seen below, the restriction of the central extension to this subgroup of
constant loops splits, hence we may assume that any projective representation
of $LG$ is associated with a true representation of $G$. By the positivity of energy,
$U$ has a lowest eigenvalue. Since the subgroup of constant loops commutes
with $U$, $G$ acts on this eigenspace. When the full representation is
irreducible, this restriction should be irreducible.

Now we can state the classification result \cite[Chapter 9]{PS}
\begin{theorem}[Pressley and Segal]
Any smooth positive-energy projective unitary representation of $LG$ is
completely reducible.
Smooth positive-energy projective unitary irreducible representations
of $LG$ can be classified by the level $h$ and the lowest weight $\lambda$
of $G$ on the lowest eigenspace of $U$. Such a representation is possible
if and only if
\begin{equation}\label{dominantgroup}
-\frac{1}{2}h\|h_\a\|^2 \le \lambda(h_\a) \le 0
\end{equation}
for each positive root $\a$ and $h_\a$ is the coroot.

All such representations are diffeomorphism covariant: Namely,
there is a projective unitary representation $U$ of $\Diff(S^1)$ such that
$U(\g)\pi(g)U(\g)^* = \pi(g\circ \g^{-1})$, where the composition
$g\circ \g^{-1}$ is again an element of $LG$.
\end{theorem}
In particular, for each level $h$ there are finitely many representations.
In terms of CFT, a representation with $\lambda = 0$ corresponds to the
vacuum representation. Each representation of $LG$ with a different level
corresponds to a different theory, and different weights corresponds to
different sectors. Finiteness of representations is the source of rationality
of these loop group models.

\subsection{The loop algebra $L\fg$}
There are similar notions of 2-cocycles, central extensions,
projective representations for Lie algebras. Since we need them also for our main
object $\sch\fg$, we explain these notions for general Lie algebras and then we state
corresponding standard results on irreducible representations of $L\fg$.

For $L\fg$, the complexification $(L\fg)_\CC$ can be
naturally defined and be identified with $L\fgc$.
It obtains a structure of $*$-Lie algebra by defining
$\xi^*(z) := (\xi(z))^*$, where in the right hand side $^*$ means the $*$-operation
with respect to the compact form.

Instead of analysing the loop algebras directly, it is customary to consider
the polynomial loops $\xi(z) = \sum_k \xi_k z^k$, where $\xi_k \in \fgc$ and
only finitely many terms appear in the sum.
Let us denote the polynomial subalgebra by $\widetilde{\fgc}$. It is easy to see that
$\widetilde{\fgc} \isom \fgc \otimes_\CC \CC[t,t^{-1}]$ with the bracket
$[x\otimes t^k,y\otimes t^l] = [x,y]\otimes t^{k+l}$.

A 2-cocycle on a complex Lie algebra $\fh$ is a bilinear form $\omega: \fh\times \fh \to \CC$
which satisfies
\begin{align}
\omega(\xi,\eta) = -\omega(\eta,\xi), \\
\omega([\xi,\eta],\zeta)+\omega([\eta,\zeta],\xi)+\omega([\zeta,\xi],\eta) = 0. \label{jacobi}
\end{align}
For a given cocycle $\omega$, we can define a new Lie algebra $\widetilde{\fh} := \fh\oplus \CC$
with the following operation,
\[
[(\xi,a_1),(\eta,a_2)] = ([\xi,\eta],\omega(\xi,\eta)),
\]
and we call it the central extension of $\fh$ by the cocycle $\omega$. It is customary to
express this algebra using a formal central element $C$ and to define the commutation relation
\[
[\xi + a_1C,\eta + a_2C] = [\xi,\eta] + \omega(\xi,\eta)C.
\]

The 2-cocycle on the algebra $L\fgc$ is unique up to a scalar \cite[Proposition 4.2.4]{PS}.
\begin{theorem}
Any $G$-invariant 2-cocycle on $L\fgc$ is proportional to the following one.
\[
\omega(\xi,\eta) = \frac{1}{2\pi i}\int_0^{2\pi} \<\xi(\theta),\eta^\prime(\theta)\> d\theta,
\]
where $\<\cdot,\cdot\>$ is the unique invariant symmetric form on $\fgc$.
\end{theorem}
At the Lie algebra level, central extensions by proportional cocycles are isomorphic, hence
we identify them and denote the equivalence class by $\widehat{\fgc}$.

As noted at the beginning of this section, it is not convenient to treat general
representations and decomposition into irreducible representations. Instead, in the following
we consider a special class of irreducible representations.

First of all, $\fgc$ has the triangular decomposition
\[
\fgc = \fn^+ \oplus \fh \oplus \fn^-,
\]
where $\fh$ is the Cartan subalgebra of $\fgc$. Let $\{H_i\}$ be the basis of
$\fh$ with respect to the root decomposition and $\omega_i \in \fh^*$ such that
$\omega_i(H_j) = \delta_{i,j}$, and $\widetilde{\a}$ be the highest root.

Following this decomposition of $\fgc$, we can decompose $\widehat{\fgc}$ as follows.
\[
\widehat{\fgc} = 
 \left(\fgc\otimes_\CC (\CC[t]\ominus \CC) \oplus \fn^+\right)
 \oplus \left(\fh \oplus \CC C\right)
 \oplus \left(\fgc\otimes_\CC (\CC[t^{-1}] \ominus \CC)\oplus \fn^-\right).
\]
This is the triangular decomposition of $\widehat{\fgc}$.
Put $\hat{\fh} := \left(\fh \oplus \CC C\right)$. We define weights on $\hat{\fh}$:
\begin{align*}
\gamma(C) = 1, \phantom{...}\gamma(H_i) = 0, \\
\widetilde{\omega}_i = \omega_i + (\omega_i,\widetilde{\a})\gamma, \\
\widetilde{\omega}_0 = \frac{1}{2}(\widetilde{\a},\widetilde{\a})\gamma.
\end{align*}

A representation of $\widehat{\fgc}$ is called a lowest weight representation
with weight $\lambda \in \hat{\fh}^*$ if there is a cyclic vector $v_0$ such that 
\begin{align*}
hv_0 = \lambda(h)v_0 &\mbox{ for } h \in \hat{\fh}, \\
x_+ v_0 = 0 &\mbox{ for } x_+ \in \fgc\otimes_\CC (\CC[t] \ominus \CC)\oplus \fn^+
\end{align*}

We have the following result \cite{Garland}.
\begin{theorem}[Garland]\label{integerlevel}
A lowest weight representation of $\widehat{\fgc}$ with weight $\lambda$ admits
a positive-definite contravariant form if and only if $\lambda$ is dominant integral,
namely, $\lambda(H_i) \in \NN$.
\end{theorem}
Furthermore, in this case the representation is unitary, and admits the action of $S^1$ as
rotation. Moreover, all such representations integrate to positive-energy projective
representations of $LG$. There is a one-to-one correspondence between dominant integral
condition above and (\ref{dominantgroup}). We have already seen that group representations
are classified by the level and the weight satisfying (\ref{dominantgroup}), so
there is a one-to-one correspondence between lowest weight irreducible representations
of $\widehat{\fgc}$ and positive-energy projective unitary irreducible representations of $LG$.

\section{Preliminaries on the Schwartz class algebra $\sch\fg$}
As noted in the introduction, we will consider an analogous problem on
infinite dimensional Lie algebras defined through the real line $\RR$, instead of $S^1$.
We identify the circle $S^1$ as the one-point compactification of the real line $\RR$
by the Cayley transform:
\[
t = i\frac{1+z}{1-z} \Longleftrightarrow z = \frac{t-i}{t+i}, \phantom{...}t \in \RR,
\phantom{..}z \in S^1 \subset \CC.
\]

The Lie algebra $\fg$ is finite dimensional, hence for a map from $\RR$ into
$\fg$ we can define the rapidly decreasing property. As one of the simplest formulations,
we take the following: Let $n$ be the dimension of $\fg$.
By fixing a basis in $\fg$, we can consider any map
$\xi:\RR \to \fg$ as the $n$-tuple of real-valued functions.
Then we say $\xi$ is rapidly decreasing if each component
of $\xi$ is rapidly decreasing. Of course this definition does not depend
on the chosen basis. It is also straightforward to define a tempered
distribution on $\sch\fg$. A linear functional $\varphi$ is said to be
tempered if each restriction of $\varphi$ to the subspaces of elements
having nonzero value only on $i$-th component is a tempered distribution.
Again this definition is independent of the choice of basis.

The main object of this paper is the following.
\begin{align*}
\sch\fg := \{\xi:\RR \to \fg, \mbox{smooth, rapidly decreasing}\}, \\
[\xi,\eta](t) := [\xi(t),\eta(t)], t \in \RR
\end{align*}
namely, the algebra of Schwartz class elements. Under the identification
of the punctured circle and the real line, it is easy to see that this algebra is
a closed subalgebra of $L\fg$. It is easy to see that as linear spaces
$\sch\fg = \fg \otimes \sch$ and the Lie algebra operation is
$[x\otimes f, y\otimes g] = [x,y]\otimes fg$.

The compact group $G$ acts on $\fg$ by the adjoint action, hence also on
$L\fg$ by the pointwise application. This action is smooth \cite[Section 3.2]{PS}.
Since $\sch\fg$ is a closed subalgebra of $L\fg$, the restricted action
of $G$ on $\sch\fg$ is also smooth. It is obvious that $\sch\fg$ is invariant under
$G$.

We are interested in positive-energy, unitary, projective representations.
Recall that for $L\fg$ we considered the subalgebra of polynomial loops and all these notions
are defined in purely algebraic terms.
For $\sch\fg$ we cannot take such an appropriate subalgebra.
Instead, we need to formulate all these properties of representations
with analytic terms from the beginning. Let $\H$ be a Hilbert space.
Note this time that $\RR$ acts on $\sch\fg$ by translation:
\[
\xi_a(t) := \xi(t-a).
\]
Again it is straightforward to define the complexification of $\sch\fg$ and
it is identified with $\sch\fgc$. The $*$-operation is naturally defined.

\begin{definition}
A projective unitary representation $\pi$ with a 2-cocycle $\omega$ of $\sch\fgc$
assigns to any element $\xi$ of $\sch\fgc$
a (possibly unbounded) linear operator $\pi(\xi)$ on $\H$ such that
there is a common dense domain $V \subset \H$ for all $\{\pi(\xi):\xi \in \sch\fgc\}$
and on $V$ it holds that
\begin{align*}
\pi([\xi,\eta])v = \left(\pi(\xi)\pi(\eta)-\pi(\eta)\pi(\xi) + \omega(\xi,\eta)\right)v, \\
\<\pi(\xi)v_1,v_2\> = \<v_1,\pi(\xi^*)v_2\>.
\end{align*}
A projective unitary representation of $\sch\fgc$ is said to have positive energy
if there is a unitary representation $U$ of $\RR$ with positive spectrum such that
$U(a)\pi(\xi)U(a)^* = \pi(\xi_a)$.

A projective unitary representation of $\sch\fgc$ is said to be smooth if for each
$v_1,v_2$ in the common domain $V$ the linear form $\<\cdot v_1,v_2\>$ is
tempered.
\end{definition}
\begin{remark}\label{invariancetrans}
Let us make some remarks.
By the same reason as in Remark \ref{invariancerot}, we can define an action of
translation on the space of 2-cocycles on $\sch\fgc$ and for a positive energy
representation the cocycle is invariant under translation.

If we have a representation of a group, it is natural to
ask if this representation produces a representation of the Lie algebra by an appropriate
derivation. And for $LG$ the answer is yes. We can prove the existence of a
common domain by utilizing finite dimensional subgroups in $LG$ with common
elements (\cite[Section 1.8]{Wassermann} or \cite[Appendix]{Carpi}).
We can define a corresponding group for $\sch\fg$,
but it is not clear if such a common domain exists for a representation of $\sch\fg$.
We will discuss on this problem in the final section.

There is also a problem on the smoothness of the representations.
As explained in the final section, in the algebraic approach to CFT
it is natural to consider the subalgebra of $\sch\fgc$ with compact support.
On the other hand, for the moment we know the proof of uniqueness of
ground state representations only for Schwartz class algebra. For the present
proof it is essential since we exploit the Fourier transforms. Unfortunately
we don't know to what extent it is natural to assume the continuity to the Schwartz
class.
\end{remark}

\section{Uniqueness of translation invariant 2-cocycle}
As we have seen in Remark \ref{invariancetrans}, for a positive-energy representation
the cocycle is always translation-invariant. Then we will restrict the
consideration to translation-invariant cocycles.
In this section, we will show that the Lie algebra $\sch\fgc$ has the unique
translation-covariant central extension.
First of all, we can define an action of $G$ on the space of cocycles by
\[
(g\omega)(\xi,\eta) := \omega(g^{-1}\xi,g^{-1}\eta).
\]
We show that we can restrict the consideration to $G$-invariant cocycles.

\begin{lemma}\label{locality}
Any 2-cocycle $\omega$ on $\sg$ is local, namely,
if $\xi$ and $\eta$ have disjoint supports, then $\omega(\xi,\eta) = 0$.
\end{lemma}
\begin{proof}
Let us take $x,y,z \in \fgc$ and $f,g,h \in \sch(\RR)$. Then by the Jacobi
identity (\ref{jacobi}),
\begin{eqnarray*}
0 &=&
\omega([x\otimes f,y\otimes g],z\otimes h) + \omega([y\otimes g,z\otimes h],x\otimes f)
+\omega([z\otimes h, x\otimes f], y\otimes g) \\
&=& \omega([x,y]\otimes fg,z\otimes h) + \omega([y,z]\otimes gh,x\otimes f) +
\omega([z,x]\otimes hf, y\otimes g).
\end{eqnarray*}
Here, let the supports of $f$ and $g$ be disjoint and compact,
and $h$ be a function such that
$h(t) = 1$ on $\supp(f)$ and $h(t) = 0$ on $\supp(g)$. Then
the equality above transforms into
\[
\omega([z,x]\otimes f,y\otimes g) = 0.
\]
Since $\fgc$ is simple, $[z,x]$ spans the whole Lie algebra $\fgc$ and the lemma is
proved by noting that these elements span elements with compact support.
\end{proof}

\begin{lemma}
Any translation-invariant continuous 2-cocycle $\omega$ on $\sg$ is equivalent up to coboundary
to a $G$-invariant cocycle.
\end{lemma}
\begin{proof}
We see that $g\omega - \omega$ is coboundary for any $g \in G$.
Since $G$ is connected, we can take a smooth path $\a$ such that
$\a(0) = e$ and $\a(1) = g$. Then by the fundamental theorem of analysis
it holds that
\begin{eqnarray*}
g\omega(\xi,\eta)-\omega(\xi,\eta) &=& \a(1)\omega(\xi,\eta) - \a(0)\omega(\xi,\eta) \\
&=& \int_0^1 \frac{d}{dt} \omega(\a^{-1}(t)\xi,\a^{-1}(t)\eta) dt.
\end{eqnarray*}

For the moment, let us assume that $\xi$ and $\eta$ have compact supports.
Then there are elements $\delta(t)$ with support compact such that
\[
\frac{d}{dt}\a^{-1}(t)\xi = [\delta(t),\a^{-1}(t)\xi],\phantom{...}
\frac{d}{dt}\a^{-1}(t)\eta = [\delta(t),\a^{-1}(t)\eta].
\]
In fact, it is enough to take an element of the form $x\otimes f$,
where $x = \a^{-1}(t)^\prime$ and $f(t)=1$ on $\supp(\xi)\cup\supp(\eta)$.

Let $\dg$ be the subalgebra of $\sg$ of elements with compact support.
We define $\g_t: \dg \to \CC$ by $\g_t(\xi) = \omega(\a^{-1}(t)\xi,\delta(t))$,
where $\delta(t)$ depends on $\xi$ as above. And this is well defined because
$\omega$ is local by Lemma \ref{locality} and the variation of $\delta(t)$ outside
the support of $\xi$ does not change $\g_t$. Then $\g_t$ is translation-invariant
since $\delta(t)$ is defined in a translation-invariant way and $\omega$ is
translation-invariant by assumption. And $\g_t$ is continuous since $\omega$ is
continuous by assumption and $\delta(t)$ is defined locally as an element
in $\dg$ and $\omega$ is local. Then $\dg$ is the finite direct sum
of test function spaces as a topological linear space, hence any translation-invariant
continuous linear functional on this space is of the form
\[
\int_\RR \psi(\xi(s))ds,
\]
where $\psi$ is a linear functional on $\fgc$. Now it is obvious that $\g_t$ extends
to $\sg$ by continuity.

As above, let $\xi,\eta$ be elements with compact support. By the continuity of
$\omega$ and the Jacobi identity (\ref{jacobi}) we see that
\begin{eqnarray*}
\frac{d}{dt}\omega(\a^{-1}(t)\xi,\a^{-1}(t)\eta)
&=& \omega([\delta(t),\a^{-1}(t)\xi],\a^{-1}(t)\eta)
+ \omega(\a^{-1}(t)\xi,[\delta(t),\a^{-1}(t)\eta]) \\
&=& -\omega([\a^{-1}(t)\xi,\a^{-1}(t)\eta],\delta(t)) \\
&=& -\omega(\a^{-1}(t)[\xi,\eta],\delta(t)) \\
&=& -\g_t([\xi,\eta]).
\end{eqnarray*}
Now this equation extends to $\sg$ since $\omega, \g,\a^{-1}$ are continuous.
In short, we have
\[
g\omega(\xi,\eta) - \omega(\xi,\eta) = -\int_0^1 \g_t([\xi,\eta]) dt,
\]
which shows that the difference between two cocycles is a linear functional
of $[\xi,\eta]$, thus it is a coboundary.

Finally, obviously the averaged cocycle
\[
\int_G g\omega dg
\]
is a $G$-invariant cocycle. And the difference
\[
\int_G (g\omega - \omega) dg
\]
is a coboundary since the integrand is a coboundary.
\end{proof}

Then we can show that the translation-invariant 2-cocycle on $\sg$ is essentially unique.
\begin{theorem}
If a translation-invariant continuous 2-cocycle $\omega$ is $G$-invariant, then $\omega(\xi,\eta)$
is proportional to the following one.
\[
\int \frac{1}{2\pi i}\<\xi(t),\eta^\prime(t)\> dt.
\]
\end{theorem}
\begin{proof}
We fix Schwartz class functions $f,g \in \mathscr{S}(\RR)$.
We define a bilinear form on $\fgc$
\[
\omega_{f,g}(x,y) := \omega(x\otimes f,y\otimes g), \phantom{...}x,y \in \fgc.
\]
Obviously, $\omega_{f,g}$ is $G$-invariant. Then, since $G$ is simple, it is known that
(see, for example, \cite[Chapter 2]{PS}) any $G$-invariant bilinear form
on $\fgc$ is proportional to the Killing form. The factor depends on
$f$ and $g$ obviously in a linear way. Hence we find
$\omega_{f,g}(x,y) = \<x,y\>\gamma(f,g)$, where $\gamma(f,g)$ is a
bilinear form on $\sch(\RR)$.

Applying the Jacobi identity (\ref{jacobi}) to three elements
$x\otimes f, y\otimes g, z\otimes h$, we see the following.
\begin{eqnarray*}
0 &=& \omega([x\otimes f, y\otimes g],z\otimes h)
 + \omega([y\otimes g, z\otimes h], x\otimes f)
 + \omega([z\otimes h, x\otimes f], y\otimes g) \\
 &=& \omega([x,y]\otimes fg, z\otimes h)
 + \omega([y,z]\otimes gh, x\otimes f)
 + \omega([z,x]\otimes hf, y\otimes g) \\
 &=& \<[x,y],z\>\g(fg,h) + \<[y,z],x\>\g(gh,f) + \<[z,x],y\>\g(hf,g).
\end{eqnarray*}
By the invariance of the Killing form, we have $-\<[x,y],z\> = \<y,[x,z]\>$ and
$\<[y,z],x\> = -\<z,[y,x]\>$. By the symmetry of the Killing form, it holds that
\[
0 = \<[x,y],z\>\left(\g(fg,h)+\g(gh,f)+\g(hf,g)\right).
\]
Then by choosing appropriate $x,y,z$ we see
\begin{equation}\label{jacobifunctions}
\g(fg,h)+\g(gh,f)+\g(hf,g) = 0.
\end{equation}

Let $f$ and $g$ be functions with disjoint supports $\supp(f) \cap \supp(g) = \emptyset$,
and let $h$ be a function such that $h(t) = 1$ on $\supp(f)$ and $h(t) =  0$ on
$\supp(g)$. By (\ref{jacobifunctions}), we have $\g(0,h)+\g(0,f)+\g(f,g) = \g(f,g) = 0$.
Namely, if the supports of $f$ and $g$ are disjoint, then $\g(f,g) = 0$. We call
this property the locality of $\g$.

Now we fix a smooth function $f$ with a compact support
$\supp(f) \subset [-\frac{a}{2},\frac{a}{2}]$. For $k \in \ZZ$, let $e_k$ be a
smooth function with a compact support such that on $[-\frac{a}{2},\frac{a}{2}]$
it holds that $e_k(t) = e^{\frac{i2\pi tk}{a}}$.

Let $g$ be some function in $\sch(\RR)$. By the locality of $\g$, the value of
$\g(f,g)$ does not depend on the form of $g$ outside the support of $f$. Then
again by the Jacobi identity for $\g$ we see $\g(fe_k,e_1)+\g(e_{k+1},f)+\g(fe_1,e_k) = 0$
or equivalently, 
\[
\g(f,e_{k+1}) = \g(fe_k,e_1) + \g(fe_1,e_k),
\]
because values of functions outside the support of $f$ do not affect
the value of $\g$. Repeating this equality replacing $f$ by $fe_1$ and $k$
by $k-1$, we have
\begin{equation*}
\g(f,e_{k+1}) = \g(fe_k,e_1) + \g(fe_k,e_1) + \g(fe_2,e_{k-1}).
\end{equation*}
It is easy to see that $\g(f,e_0) = 0$. By induction it holds for $k\ge1$ that
\begin{equation}\label{induction}
\g(f,e_k) = k\g(fe_{k-1},e_1).
\end{equation}
A similar argument holds also for $k \le 0$.

We define $\varphi(f) := \g(f,e_1)$. By the translation-invariance of $\g$,
we see that $\varphi(f) = c_a\int e^{\frac{i2\pi t}{a}}f(t)dt$ for some constant $c_a \in \CC$.
Then by the equality (\ref{induction}) we have
\[
\g(f,e_k) = c_a \int ke^{\frac{i2\pi tk}{a}}f(t)dt.
\]
Then for a smooth function $g$ with support in $[-\frac{a}{2},\frac{a}{2}]$,
by considering its Fourier expansion $g(t) = \sum_k e^{\frac{i2\pi tk}{a}}g_k$, it holds that
\[
\g(f,g) = \frac{c_a a}{2\pi i} \int g^\prime(t)f(t)dt.
\]
But the interval $[-\frac{a}{2},\frac{a}{2}]$ is in reality arbitrary,
then $c_a a$ does not depend on $a$ and
this equality holds for any compact support functions. Then by the
continuity of $\g$ it holds also for Schwartz class functions.
\end{proof}

We take the following as the standard normalization.
\[
\omega_1(\xi,\eta) := \frac{1}{2\pi i}\int \<\xi(t),\eta^\prime(t)\>dt.
\]
We say that a positive-energy representation has level $c$ if
its cocycle is $c\omega_1$.

\section{Uniqueness of ground state representations}\label{uniqueness}
First of all, let us specify the class of representation in which we are
interested.
\begin{definition}
If a smooth positive-energy unitary projective representation $\pi$ of
$\sg$ on the common domain $V \subset \H$ has a unique vector $\Omega$
(up to scalar) such
that $\Omega$ is invariant under the unitary implementation $U$ of the
translation and $V$ is algebraically generated by $\Omega$,
then $\pi$ is said to be a ground state representation.
\end{definition}

Throughout this section, $\pi$ is a ground state representation of
$\sg$ on $\H$, with a common domain $V$, $\Omega$ is the ground state
vector, and $U$ is the one-parameter group of unitary operators which
implements the translation.

Note that any vacuum representation of $L\fgc$ is a ground state representation.
Any vacuum representation of $L\fgc$ with a different value of
the cocycle corresponds to a different conformal field theory.
As the title of the present article suggests,
we want to show the uniqueness of ground state
for a CFT. In other words, any ground state representation of $\sch\fgc$ with
a fixed cocycle is the vacuum representation.

Note also that we assume from the beginning that the ground state vector
$\Omega$ is cyclic and unique. Since we need to treat unbounded operators,
it is not convenient to discuss decomposition of representations.
We will return to this point in the final section.

Let us start with several observations similar
to the classical argument in \cite{FST}, which is originally given by
L\"uscher and Mack in their unpublished article.
Let $E$ be the spectral measure associated with $U$.
If $g$ is a smooth bounded function on $\RR$, we denote by $g(U)$
the functional calculus associated with $E$, defined by
\begin{align*}
U(a) = \int e^{i2\pi pa} dE(p) \mbox{ for } a \in \RR, \\
g(U) = \int g(p) dE(p).
\end{align*}

\begin{lemma}\label{spectralmeasure}
If the Fourier transform $\hat{f}$ of $f \in \sch$ has support
in $\RR_+$, then it holds for any $x \in \fgc$ that $\pi(x\otimes f)\Omega = 0$.
\end{lemma}
\begin{proof}
Recall that the Fourier transform is a homeomorphism of the space of Schwartz
class functions $\sch$. This holds also true for $\sg$, since it is just the
space of Schwartz class functions with several components. So we can define
a Fourier transform $\hat{\xi}$ of an element $\xi \in \sg$ as an element
in $\sg$ with the Fourier transformed functions in each component.
To keep the notation simple, let us define $\hat{\pi}$
the Fourier transform of $\pi$, namely  $\hat{\pi}(\hat{\xi}) := \pi(\xi)$.

The action of translation on $\sg$ is as follows: $\xi_a(t) = \xi(t-a)$.
In Fourier transform, it becomes
\[
\widehat{\xi_a}(p) = \int e^{-i2\pi pt}\xi(t-a) dt = e^{-i2\pi pa}\hat{\xi}(p).
\]
We introduce an obvious notation $g\xi(t) := g(t) \xi(t)$ where $g$ is a smooth
function on $\RR$ and $\xi \in \sg$. Then letting $e_a(p) := e^{i2\pi pa}$, we can
write the relation above as $\widehat{g}_a = e_{-a}\hat{g}$.
Let $U$ be the unitary operators implementing translation.
By the invariance of $\Omega$, we can write this as follows.
\[
U(a)\pi(\xi)\Omega = U(a)\pi(\xi)U(a)^*\Omega = \pi(\xi_a)\Omega
= \hat{\pi}(e_{-a}\hat{\xi})\Omega.
\]

Now let $x$ and $f$ be as in the statement and
let $g$ be a function in $\sch$ such that its Fourier transform
has $\hat{g}(p) = 1$ on $\supp(\hat{f})$ and has support in $[-\frac{S}{2},\frac{S}{2}]$,
where $S$ is some positive number.
The restriction of $\hat{g}$ to $[-\frac{S}{2},\frac{S}{2}]$ can be expanded into a Fourier
series
\[
\hat{g}(p) = \sum_{k\in\ZZ} e^{\frac{i2\pi kp}{S}}g_{S,k}.
\]
Recall that the convergence of the Fourier series is smooth
(uniform on $[-\frac{S}{2},\frac{S}{2}]$ for each derivative).
If $p$ is in the interval $[-\frac{S}{2},\frac{S}{2}]$, then it holds that
\[
\hat{f}(p) = \hat{f}(p)\hat{g}(p)
= \hat{f}(p)\left(\sum_{k\in\ZZ} e_{\frac{k}{S}}(p)g_{S,k}\right)
= \sum_{k\in\ZZ}\hat{f}(p)e_{\frac{k}{S}}(p)g_{S,k},
\]
and the convergence in the last series is still smooth on $[-\frac{S}{2},\frac{S}{2}]$,
since $\hat{f}$ is a smooth function with a compact support in this
interval, so the Leibniz rule shows the convergence. Then,
looking at only the left and right hand sides we see that
the equality above holds on the whole real line, simply because
$\hat{f}(p) = 0$ outside the interval $[-\frac{S}{2},\frac{S}{2}]$. The convergence
is still smooth.

Since $\pi$ is an operator valued distribution, so is $\hat{\pi}$,
which is weakly continuous with respect to the smooth topology on $\sg$.
Then we find
\begin{eqnarray*}
\pi(x\otimes f)\Omega &=& \hat{\pi}(x\otimes \hat{f})\Omega \\
&=& \sum_{k\in\ZZ} \hat{\pi}(x\otimes \hat{f}e_{\frac{k}{S}}g_{S,k})\Omega.
\end{eqnarray*}

On the other hand, as a function on the whole real line $\RR$,
the series
\[
\sum_{k\in\ZZ} e^{\frac{i2\pi kp}{S}} g_{S,k}
\]
is uniformly convergent, since it is uniformly convergent on
an interval $[-\frac{S}{2},\frac{S}{2}]$ because it is the Fourier expansion of $\hat{g}$, and
uniformly convergent also on any translation of the interval $[-\frac{S}{2},\frac{S}{2}]$ since the series is
obviously a function with a period $S$.
It holds that
\[
\sum_{k\in\ZZ} e^{\frac{-i2\pi kp}{S}} g_{S,k}
= \sum_{k\in\ZZ} e^{\frac{i2\pi kp}{S}} g_{S,-k},
\]
and it is also uniformly convergent.
Let $g_S$ be the function
which has the series above as the Fourier transform.
Then the series of operators $\sum_{k\in\ZZ} U\left(\frac{k}{S}\right)g_{S,-k} = g_S(U)$
is strongly convergent. Applying this equality to the vector $\pi(x\otimes f)\Omega$
we have 
\begin{eqnarray*}
g_S(U)\pi(x\otimes f)\Omega &=&
\sum_{k\in\ZZ} U\left(\frac{k}{S}\right)g_{S,-k} \pi(x\otimes f)\Omega \\
&=& \sum_{k\in\ZZ} g_{S,-k}U\left(\frac{k}{S}\right)
\pi(x\otimes f)U\left(-\frac{k}{S}\right)\Omega \\
&=& \sum_{k\in\ZZ} g_{S,-k}\pi\left(x\otimes f_{\frac{k}{S}}\right)\Omega,
\end{eqnarray*}
since $\Omega$ is invariant under translation and $U$ implements it.
Then by replacing $k$ by $-k$ we can write it as follows.
\begin{eqnarray*}
g_S(U)\pi(x\otimes f)\Omega
&=& \sum_{k\in\ZZ} g_{S,k}\pi\left(x\otimes f_{-\frac{k}{S}}\right)\Omega \\
&=& \sum_{k\in\ZZ} g_{S,k}\hat{\pi}\left(x\otimes e_{\frac{k}{S}}\hat{f}\right)\Omega \\
&=& \hat{\pi}(x\otimes \hat{g}\hat{f})\Omega \\
&=& \hat{\pi}(x\otimes \hat{f})\Omega \\
&=& \pi(x\otimes f)\Omega.
\end{eqnarray*}

If we let $S$ tend to $\infty$, $g_S(U)$ tends to an operator $\tilde{g}(U)$,
where $\tilde{g}$ has the Fourier transform $\hat{g}(-p)$.
Now recall that the condition on $g$ is that its Fourier transform $\hat{g}$
has compact support and is equal to $1$ on the support of $\hat{f}$.
Then $\hat{\tilde{g}}$ is equal to $1$ on $-\supp(\hat{f})$ and for such $\tilde{g}$
it holds $\tilde{g}(U)\pi(x\otimes f)\Omega = \pi(x\otimes f)$. Then the support of
spectral measure of the vector $\pi(x\otimes f)\Omega$ with respect to $U$ must
be contained in $-\supp(\hat{f})$.

In particular, if $\supp(\hat{f})$ is compactly supported in $\RR_+$, then
the spectral measure of $\pi(x\otimes f)\Omega$ is compactly supported in $\RR_-$,
hence it is equal to $0$ because of the positivity of the energy.
Any function with support in $\RR_+$ is smoothly approximated by a function
compactly supported in $\RR_+$, so the continuity of $\pi$ as an operator
valued distribution completes the lemma.
\end{proof}

Let us define $\psi(\xi) := \<\pi(\xi)\Omega,\Omega\>$. By definition $\Omega$ is
unique for ground state representations, hence $\psi$ is an invariant for
this class of representations.
\begin{lemma}\label{support}
$\psi(\xi)$ depends only on $\hat{\xi}(0) \in \fgc$.
\end{lemma}
\begin{proof}
We fix $x \in \fgc$ and consider the restriction
\begin{align*}
\psi_x:\sch &\longrightarrow \CC \\
       f    &\longmapsto  \psi(x\otimes f).
\end{align*}
It is obvious that $\psi_x$ is invariant under translation.
Hence it has the form $\psi_x(f) = C_x\hat{f}(0)$, where $C_x$ is a constant
depending on $x$.
The linear functional $\psi$ can be reconstructed by such
restrictions, hence $\psi$ itself depends only on $\hat{\xi}(0)$.
\end{proof}

\begin{lemma}\label{scalar}
Let $\{\xi_n\}$ be a sequence of elements in $\sg$ such that
\begin{itemize}
\item each $\hat{\xi}_n$ has a compact support.
\item for $p \ge 0$, $\hat{\xi}_n(p) = \hat{\xi}_m(p)$ for any $n, m \in \NN$.
\item for $p < 0$, the norm of $\hat{\xi}_n(p) \in \fgc$ with respect to the Killing form
is uniformly bounded and the Lebesgue measure
of $\supp(\hat{\xi}_n)\cap \RR_-$ tends to $0$.
\end{itemize}
Then $\pi(\xi_n)\Omega$ is convergent to $\<\pi(\xi_n)\Omega,\Omega\>\Omega$.
\end{lemma}
\begin{proof}
By the proof of Lemma \ref{spectralmeasure}, $\pi(\xi_n)\Omega$ is contained in
$\chi_{-\supp(\hat{\xi}_n)}(U)\H$. The intersection of these spaces is clearly
the one-dimensional space $\CC\Omega$.
To see the convergence, we have to estimate the following.
\begin{eqnarray*}
\|\pi(\xi_m-\xi_n)\Omega\|^2 &=& \<\pi(\xi_m-\xi_n)^*\pi(\xi_m-\xi_n)\Omega,\Omega\> \\
&=& \<\pi(\xi_m-\xi_n)\pi(\xi_m-\xi_n)^*\Omega,\Omega\> + \<[\pi(\xi_m-\xi_n)^*,\pi(\xi_m-\xi_n)]\Omega,\Omega\>.
\end{eqnarray*}
The first term vanishes by Lemma \ref{spectralmeasure}.

We can transform the second term
using the commutation relation and obtain
\[
\<\pi([(\xi_m-\xi_n)^*,\xi_m-\xi_n])\Omega,\Omega\> - \omega((\xi_m-\xi_n)^*,\xi_m-\xi_n).
\]
Let us estimate the first term of this difference.
By Lemma \ref{support}, it is enough to estimate the value at $0$ of the Fourier transform
of $[(\xi_m-\xi_n)^*,\xi_m-\xi_n]$.
By the assumption, the Fourier transform of $\xi_m-\xi_n$ is also bounded and the measure 
of its support tends to $0$ as $m,n$ tend to $\infty$. In general we have
\[
\widehat{[\eta^*,\eta]}(0) = \int [\widehat{\eta^*}(p),\widehat{\eta}(-p)]dp.
\]
If we apply this to $\eta = \xi_m-\xi_n$, the integral is bounded by
(the square of the double of) the uniform bound of $\{\hat{\xi}_m\}$, the norm of the commutator of $\fgc$
and the measure of the support of $\hat{\xi}_m-\hat{\xi}_n$. Then this tends to $0$.
By continuity of $\pi$, this term tends to $0$.
For the second term, we know the concrete form of the cocycle $\omega$ and
in the Fourier transform it takes
\[
\omega(\eta^*,\eta) = \frac{c}{2\pi i}\int i2\pi p\<\widehat{\eta^*}(-p),\widehat{\eta}(p)\>dt,
= c\int p\<\widehat{\eta^*}(-p),\widehat{\eta}(p)\>dt
\]
then by a similar reasoning the corresponding term converges to $0$.

Now that we know that the concerned sequence converges to a scalar multiple of $\Omega$,
it is enough to determine the coefficient $\<\pi(\xi_n)\Omega,\Omega\>$.
By Lemma \ref{support} this is determined by $\hat{\xi}(0)$ and by the assumption
this is constant.
\end{proof}

\begin{lemma}\label{commutator}
Let $\{\xi_n\}$ be a sequence of elements in $\sg$.
Assume that components of $\hat{\xi}_n$ are uniformly bounded and convergent
to a bounded function and
the Lebesgue measure of the support of $\hat{\xi}_{n'}-\hat{\xi}_n$ is monotonically decreasing
to $0$. Then for any $\eta \in \sg$, the commutator $[\xi_n,\eta]$ is smoothly
convergent to an element in $\sg$.
\end{lemma}
\begin{proof}
It is enough to consider the case where $\xi_n = x\otimes f_n, \eta = y\otimes g$,
since the general case is a finite linear combination of such elements and the
convergence of the commutator follows immediately.
In this case, the commutator is expressed with the Fourier transform as follows.
\[
\widehat{[\xi_n,\eta]}(p) = [x,y]\otimes \int \hat{f}_n(s)\hat{g}(p-s)ds.
\]
The convergence in the smooth topology is defined as the uniform convergence
of the following functions.
\[
p^l\widehat{[\xi_n,\eta]}^{(m)}(p) = [x,y]\otimes \int \hat{f}_n(s)p^l\hat{g}^{(m)}(p-s)ds.
\]
Since the function in the integrand is uniformly bounded, and the measure of the support
$\hat{\xi}_{n'}-\hat{\xi}_n$ is decreasing to $0$, the integral is convergent uniformly.
\end{proof}

\begin{lemma}\label{characterization}
The representation $\pi$ is characterized by $\psi$ and the level $c$. Namely,
any two representations which correspond to the same functional $\psi$ and the 
same level $c$ are unitarily equivalent.
\end{lemma}
\begin{proof}
We will show that the $n$-point function
$\<\pi(\xi_1)\pi(\xi_2)\cdots\pi(\xi_n)\Omega,\Omega\>$
is determined by $\psi$ and $c$ for any $n$.
Since $\Omega$ is cyclic for $\pi$, this implies that
any inner product of the form
$\<\pi(\xi_1)\cdots\pi(\xi_n)\Omega,\pi(\eta_1)\cdots\pi(\eta_m)\Omega\>$ is
determined by $\psi$ and $c$. If two representations $\pi_1, \pi_2$ have the
same $\psi$ and $c$, then the map
\[
\pi_1(\xi_1)\cdots\pi_1(\xi_n)\Omega \longmapsto \pi_2(\xi_1)\cdots\pi_2(\xi_n)\Omega
\]
is a unitary map intertwining the two representations
since by the definition of ground state representation these vectors span
the dense common domain $V$. Furthermore,
by the continuity of $\pi$, we may assume that $\{\hat{\xi}_k\}$ have
compact supports.

We show that for $n \ge 2$, the $n$-point function is reduced to
$(n-1)$-point functions. Then an induction about $n$ completes the proof.
Let us decompose $\xi_1$ into two parts $\xi_1 = \xi_+ + \xi_-$ such
that $\hat{\xi}_+$ has support in $\RR_+$.
By the Lemma \ref{spectralmeasure} we know that
$\pi(\xi_+)\Omega = 0$. In the $n$-point function,
we can take $\pi(\xi_+)$ to the right using the commutation relation and annihilate
it letting it act on $\Omega$, so that the $n$-point function will be
reduced to the sum of $(n-1)$-point functions and $\xi_-$ part. Explicitly,

\begin{eqnarray*}
\<\pi(\xi_1)\pi(\xi_2)\cdots\pi(\xi_n)\Omega,\Omega\>
&=& \<\pi(\xi_+ + \xi_-)\pi(\xi_2)\cdots\pi(\xi_n)\Omega,\Omega\> \\
&=& \<\pi(\xi_-)\pi(\xi_2)\cdots\pi(\xi_n)\Omega,\Omega\> \\
& & + \<[\pi(\xi_+),\pi(\xi_2)]\cdots\pi(\xi_n)\Omega,\Omega\> \\
& & + \<\pi(\xi_2)\pi(\xi_+)\cdots\pi(\xi_n)\Omega,\Omega\> \\
&=& \<\pi(\xi_-)\pi(\xi_2)\cdots\pi(\xi_n)\Omega,\Omega\> \\
& & + \sum_k \<\pi(\xi_2)\cdots[\pi(\xi_+),\pi(\xi_k)]\cdots\pi(\xi_n)\Omega,\Omega\>. \\
\end{eqnarray*}

This is equal to the following since $\pi$ is a projective representation.
\begin{eqnarray*}
\<\pi(\xi_1)\pi(\xi_2)\cdots\pi(\xi_n)\Omega,\Omega\>
&=& \<\pi(\xi_-)\pi(\xi_2)\cdots\pi(\xi_n)\Omega,\Omega\> \\
& & + \sum_k \<\pi(\xi_2)\cdots(\pi([\xi_+,\xi_k])-\omega(\xi_+,\xi_k))\cdots\pi(\xi_n)\Omega,\Omega\>. \\
\end{eqnarray*}

Now, let $f_\epsilon$ be a smooth function such that
$f_\epsilon(p) = 1$ for $p \ge \epsilon$, $f_\epsilon(p) = 0$ for $p \le 0$ and
$0 < |f_\epsilon| < 1$ for $0 < p < \epsilon$.
Let us make a decomposition of $\xi_1$ such that
$\hat{\xi}_{\epsilon+}(p) = f_\epsilon(p)\hat{\xi}(p)$ and
$\xi_{\epsilon-} = \xi_1-\xi_{\epsilon+}$.

On the one hand, $\xi_{\epsilon+}$
satisfies the assumption of Lemma \ref{commutator}, hence by letting
$\epsilon$ tend to $0$, all the brackets above are convergent
to images of some elements in $\sg$, hence there appear images by $\pi$ and
scalar multiples of $c$ which depends only on the Lie algebra structure.
On the other hand, $\xi_{\epsilon-}^*$ satisfies
the assumption of Lemma \ref{scalar} and
$\pi(\xi_{\epsilon-}^*)\Omega = \pi(\xi_{\epsilon-})^*\Omega$ is convergent
to $\psi(\xi_{\epsilon-}^*)\Omega$ (which does not depend on $\epsilon$).
This reduces every term in the $n$-point function to $(n-1)$-point functions,
$\psi$ and $c$.
\end{proof}

We have seen that $\psi$ and $c$ characterize the representation
$\pi$. Finally we show that $\psi$ is not necessary and
$\pi$ is determined only by $c$.
\begin{theorem}\label{psizero}
For any ground state representation, $\psi = 0$, thus $\<\pi(\xi)\Omega,\Omega\> = 0$ for
any $\xi \in \sch\fgc$.
\end{theorem}
\begin{proof}
We will show this by contradiction. To be precise, we assume that $\psi \neq 0$ and
we show that representation is not unitary.

By definition it is easy to see that $\psi$ is self-adjoint.
Let $\xi \in \sg$. As we have seen in Lemma \ref{support}, $\psi(\xi)$
is determined only by $\hat{\xi}(0)$. Let us define $\psi_0:\fgc \to \CC$
such that $\psi(\xi) = \psi_0(\hat{\xi}(0))$.

By the assumption, there is an element $x$ from $\fgc$ such that $\psi_0(x) \neq 0$.
Since $\psi$ is self-adjoint, so is $\psi_0$ and we may assume that $x$ is
self-adjoint and $\psi_0(x) \in \RR$. Then there is a Cartan subalgebra which contains $x$.
Let us consider the root decomposition of $\fgc$ with respect to this
Cartan subalgebra. Let $\a$ be an element in the root system $\Psi$,
and let $\mathfrak{sl}(\a)$ be
the subalgebra of $\fgc$ isomorphic to $\mathfrak{sl}_2(\CC)$ associated
to $\a$. We define put $E_\a, F_\a, H_\a$ the elements in $\mathfrak{sl}(\a)$
corresponding to
\[
E=\left(\begin{array}{cc} 0 & 1 \\ 0 & 0 \end{array}\right),
F=\left(\begin{array}{cc} 0 & 0 \\ 1 & 0 \end{array}\right),
H=\left(\begin{array}{cc} 1 & 0 \\ 0 & -1 \end{array}\right).
\]
We may assume there is a root $\a$ such that $\psi_0(H_\a) \neq 0$, since
the Cartan subalgebra is spanned by $\{H_\a\}_{\a \in\Psi}$.

As the first case, we assume $\psi_0(H_\a)  > 0$.
It holds that $[E_\a,F_\a] = H_\a$ and $E_\a^* = F_\a$.
Let us take a smooth real function $f \in \sch(\RR)$ with $\supp(f) \subset \RR_-$.
We will find a vector in $\H$ with negative norm. In fact, it holds that
\begin{eqnarray*}
\|\pi(E_\a\otimes f)\Omega\|^2
&=& \<\pi(E_\a\otimes f)\Omega, \pi(E_\a\otimes f)\Omega\> \\
&=& \<[\pi(F_\a\otimes \overline{f}), \pi(E_\a\otimes f)]\Omega,\Omega\> \\
&=& \<\pi([F_\a,E_\a]\otimes |f|^2)\Omega\>
- \<F_\a,E_\a\>\frac{c}{2\pi i}\int \overline{f(t)}f^\prime(t) dt \\
&=& \psi_0(-H_\a) \int |\hat{f}(p)|^2dp
- c\<F_\a,E_\a\>\int p|\hat{f}(p)|^2dp .
\end{eqnarray*}
Then if we take a function $f$ such that $\hat{f}$ has support
sufficiently near to $0$ but nonzero, then the norm must be
negative.

If $\psi_0(H_\a) < 0$, we only have to consider the norm
of $\pi(F_\a\otimes f)\Omega$.
\end{proof}

\begin{corollary}
All the ground states on $\sg$ are completely classified by
$c$ and such a representation is possible if and only if
$c \in \NN_+$.
\end{corollary}
\begin{proof}
We have seen in Lemma \ref{characterization} that ground states on $\sg$
are completely classified by $\psi$ and $c$, on the other hand Lemma
\ref{psizero} tells us that only the case $\psi = 0$ is possible.

For $L\fgc$ we know that lowest weight representations with invariant
vector with respect to the M\"obius group M\"ob are completely classified
by $c$ and the only possible values of $c$ are positive integers.
What remains to prove is that every ground state representation of
$\sg$ extends to $L\fgc$. This is done by the repetition of the
argument by \cite[Section 4]{BS}.
In fact, we know that a ground state representation $\pi$
is determined by the value of $c$, and the cocycle $\omega$ is invariant
under dilation. Also the positive and negative parts decomposition
in the proof of Lemma \ref{characterization} is not affected by dilation.
Then it is straightforward to check that there is a unitary representation
of dilation under which $\Omega$ is invariant and $\pi$ is covariant.
By analogy with the L\"uscher-Mack theorem, all $n$-point functions extend to
the circle $S^1$ and turn out to be invariant under M\"ob.
Then by the reconstruction theorem, we obtain the representation of $L\fgc$.

It is known that in this case the level $c$ must be a positive integer
by Theorem \ref{integerlevel} (due to \cite{Garland}) or \cite{PS}.
\end{proof}

\section{Concluding remarks: ground states of conformal nets}
\subsection{Ground state representations of Lie algebras and
conformal nets}

The main result in the previous section is the uniqueness of
ground state on the Schwartz class subalgebra $\sg$ of
$L\fgc$. Here we explain its (possible) physical implication.

In the operator-algebraic approach to one-dimensional chiral
conformal field theory, the mathematical object in which we are
interested is the following. Let $\mathfrak{J}$ be the set of
open connected sets of the circle $S^1$. The M\"obius group
M\"ob $\isom PSL(2,\RR)$ acts on $S^1$.
Let $\H$ be a (separable) Hilbert space. A conformal net $\A$
of von Neumann algebras is an assignment to each interval $I \in \mathfrak{J}$
of a von Neumann algebra $\A(I)$ acting on $\H$ which satisfy the following
conditions.
\begin{enumerate}
\item If $I_1 \subset I_2$, then $\A(I_1) \subset \A(I_2)$ (isotony).
\item If $I_1 \cap I_2 = \emptyset$, then $\A(I_1)$ and $\A(I_2)$ commute (locality).
\item There is a unitary representation $U$ of M\"ob such that
$U(\varphi)\A(I)U(\varphi)^* = \A(\varphi(I))$ for $\varphi$ in M\"ob (covariance).
\item The restriction of $U$ to the subgroup of rotations in M\"ob
has positive spectrum (positivity of energy).
\item There is a vector $\Omega$ which is invariant under $U$ and cyclic for
$\A(I)$ (existence of vacuum).
\end{enumerate}

In our context, examples of conformal nets are given in terms of vacuum representations
of loop groups (see for example \cite{GF}).
Explicitly, let $LG$ be the loop group of a certain simple simply connected
Lie group $G$. We take a positive-energy vacuum representation $\pi$ of $LG$ at certain level $k$. Then
we set
\[
\A_{LG_k}(I) := \{\pi(g):\supp(g) \subset I\}^{\prime\prime}.
\]
Isotony is obvious from the definition. Locality comes from the locality of cocycle.
Each such vacuum representation is covariant under the diffeomorphism group $\mathrm{Diff}(S^1)$,
in particular under M\"ob. Positivity of energy is readily seen. The lowest eigenvector of
rotation behaves as the vacuum vector.

A conformal net is considered as a mathematical realization of a physical model.
Several physical states are realized as states on the quasilocal $C^*$-algebra
\[
\overline{\bigcup_{I\Subset \RR} \A(I)}^{\|\cdot\|}.
\]
We denote it simply by $\A$.
On this $C^*$-algebra, the the group of translations acts as one-parameter automorphism $\tau$.

Among all states on $\A$, states which represent thermal equilibrium are of particular interest.
The property of thermal equilibrium is characterized by the following KMS condition \cite{BR2}.
\begin{definition}
A state $\varphi$ on a $C^*$-algebra $\A$ is called a $\beta$-KMS state
(with respect to a one-parameter automorphism group $\tau$)
if for each pair $x,y \in \A$ there is an analytic function
$f(z)$ on $0 < \mathrm{Im} z < \beta$
and continuous on $0 \le \mathrm{Im} z \le \beta$ such that it holds for $t \in \RR$
\[
f(t) = \varphi(x\tau_t(y)), \phantom{...}f(t+i\beta) = \varphi(\tau_t(y)x).
\]
Here, $\frac{1}{\beta}$ is interpreted as the temperature of the state of equilibrium.
\end{definition}

As easily seen, when the temperature goes to $0$, $\beta$ goes to the infinity and
the domain of analyticity approaches to the half-plane. We simply take the following
definition, and consider it as an equilibrium state with temperature zero.
\begin{definition}
A state $\varphi$ on a $C^*$-algebra $\A$ is called a ground state with respect to $\tau$
if for each pair $x,y \in \A$ there is an analytic function $f(z)$ on $0 < \mathrm{Im} z$
and continuous on $0 \le \mathrm{Im} z$ such that it holds for $t \in \RR$
\[
f(t) = \varphi(x\tau_t(y)).
\]
\end{definition}

In general, if $\varphi$ is invariant under translation, the action of translation
is implemented canonically by a one-parameter group of unitary operators in its GNS representation.
It is known that this condition is characterized by its property in the GNS representation
\cite{BR2}.
\begin{theorem}
A translation-invariant state $\varphi$ on $\A$ is a ground state if and only if
the generator of translation in the GNS representation has positive spectrum.
\end{theorem}

No direct and obvious way to classify ground states on general conformal
nets is at hand, but the case of loop group nets seems rather hopeful.
Let $\pi$ be a vacuum representation of loop group $LG$,
$\varphi$ be a ground state on $\A_{LG_k}$ and $\pi_\varphi$ be
the GNS representation with respect to $\varphi$.
Let us call temporarily $\mathscr{D}G$ the subgroup of $LG$ with elements
compactly supported in $\RR$, with the identification of $\RR$ as a
part of $S^1$.
Since $\A_{LG_k}$ is generated by local operators, for any group element
$\xi \in LG$ with support in $\RR$ we have $\pi(\xi) \in \A_{LG_k}$,
hence the composition $\pi_\varphi\circ\pi$ is a representation of
$\mathscr{D}G$, covariant under translation implemented by one-parameter
unitary group with positive generator, containing a translation-invariant
vector. Then to classify all ground state representations of $\A_{LG_k}$,
it is enough to classify ground state representations of $\mathscr{D}G$.

Hence the result of this paper can be considered as a first step towards
the classification of ground states of loop groups. The remaining steps
should be roughly the following.
\begin{itemize}
\item To show that every ground state representation of $\mathscr{D}G$ is
differentiable and induces a representation of $\mathscr{D}\fg$.
\item To show that every ground state representation of $\mathscr{D}\fg$ can
be extended to $\sch\fg$.
\end{itemize}
Combining it with the uniqueness result of this paper we would see
the uniqueness of ground state on $\A_{LG_k}$.

Unfortunately, the author is not aware of any concrete strategy to
these points. Recently a general theory about differentiability of
representations of infinite dimensional groups was established by
Neeb \cite{Neeb}. Detailed analysis for ground state representations
could lead to general differentiability. For the second point,
invariance of the ground state vector could imply the extension
of operator valued distribution, in analogy of the case of
distribution.

\subsection{Irreducibility and factoriality of representations}
In section \ref{uniqueness} we have classified representations of
$\sg$ with a cyclic ground state vector. We need to justify that
the assumption of cyclicity is not essential. In fact we would like
to show that any ground state representation should be decomposed
into representations with cyclic ground state vector. This is a bit
problematic at the level of Lie algebras, because operators are unbounded,
hence not defined on the whole space. We would have to take care of
commutation of unbounded operators, density of domain, existence of
eigenvalues, etc. Instead, we content ourselves with considering
the decomposition problem at the level of conformal net.

Here we just restate some well-known results, mainly taken from
standard textbooks. The first two results come from \cite[1.2.3 Corollary and 1.2.7 Corollary,
respectively]{Baumgaertel}.
\begin{theorem}
A ground state representation with a unique ground state vector
is irreducible.
\end{theorem}
\begin{theorem}
Let $\pi$ be a ground state representation. If $\bigvee_{I \Subset \RR} \pi(\A(I))$
is a factor then necessarily it is $B(\H)$. In this case the ground state vector is
unique.
\end{theorem}
We remark that in the book \cite{Baumgaertel}, the statements are given for
``vacuum representations'' of two or higher dimensional Poincar\'e
covariant nets of von Neumann algebras. In reality in the proofs of these results,
covariance with respect to Lorentz transformations is not used, and
adaptation to the one-dimensional case is straightforward.

For the following we refer the book \cite[Section 5.3.3]{BR2}.
\begin{theorem}
The following are equivalent.
\begin{enumerate}
\item The set of ground states is simplex.
\item If the von Neumann algebra generated by the GNS representation of a
ground state is a factor, then it is $B(\H)$.
\end{enumerate}
\end{theorem}

Then, a general ground state can be decomposed uniquely into
extremal states (see \cite[Theorem 4.1.15]{BR1}). Any extremal states has a factorial
representation (in fact, if the GNS representation is not factorial then a nontrivial
central projection commutes with the representatives of translation \cite[Theorem 1.1.1]{Baumgaertel},
thus the GNS vector decomposes into two ground state vectors),
hence by the previous theorem it has a unique cyclic ground state vector.

Thanks to these general results, we can reduce a general ground state
into a convex combination of pure ground states. A pure ground state
has a unique ground state vector in its GNS representation. To
classify ground states it is enough to find all pure ground states.
Then it is natural to restrict also the study of Lie algebra representations
to the case with a unique cyclic ground state vector.

\subsubsection*{Acknowledgment.}
I would like to thank Roberto Longo for his inspiring discussions and constant support.
I am grateful to Victor Kac, Karl-Henning Rehren, Mih\'aly Weiner, Paolo Camassa, Daniela
Cadamuro and Wojciech Dybalski for their useful comments, and to the referee of
Annals Henri Poincar\'e for the careful reading and the detailed suggestions.

\end{document}